\documentclass[10pt]{article}

\usepackage{amsmath,amsthm}
\usepackage{amssymb}
\usepackage{url}
\usepackage{srcltx}
\usepackage{enumerate}
\usepackage[latin1]{inputenc}
\usepackage{mathrsfs}

\newcommand{\V}[2]{V_{#1,#2}}
\def\Vui{\V{u}{i}}
\newcommand{\G}[2]{G_{#1,#2}}
\def\Gui{\G{u}{i}}

\newcommand{\set}[1]{\{#1\}}
\newcommand{\Vuid}[0]{\Vui^{id}}

\newcommand{\cla}[1]{\mathscr #1}
\newcommand{\clac}[0]{\cla{C}}
\newcommand{\clak}[0]{\cla{K}}
\newcommand{\simdif}[0]{\mathbin{\Delta}}

\def\td{\text{td}} 
\def\clos{\text{clos}}

\renewenvironment{proof}{\par \noindent \textbf{Proof.}
}{\hfill$\Box$\medskip}
\newenvironment{proofof}[1]{\par \noindent \textbf{Proof of #1.}
}{\hfill$\Box$\medskip} 

\newtheorem{theorem}{Theorem}
\newtheorem{corollary}[theorem]{Corollary}

\newtheorem*{claim}{Claim}
\newtheorem{proposition}[theorem]{Proposition}

\date{\today}

 \title{Locally identifying coloring in bounded expansion classes of
    graphs \footnote{This work was partially supported by the ANR grant
    EGOS 12 JS02 002 01.}}

  \author{Daniel Gon\c{c}alves\thanks{LIRMM - Univ. Montpellier 2, CNRS - 161 rue Ada, 34095
    Montpellier Cedex 5, France}, Aline Parreau\thanks{LIFL - Univ. Lille 1, INRIA - Parc scientifique de la haute borne,
    59650 Villeneuve d'Ascq, France}, Alexandre Pinlou{$^\dagger$\!\!} \thanks{Second affiliation: D\'epartement de Math\'ematiques et
    Informatique Appliqu\'es, 
    Universit\'e Paul-Val\'ery, Montpellier 3, Route de Mende, 34199
    Montpellier Cedex~5, France.}}
  
\begin{document}

\maketitle

\begin{abstract}
  A proper vertex coloring of a graph is said to be \emph{locally
    identifying} if the sets of colors in the closed neighborhood of
  any two adjacent non-twin vertices are distinct. The lid-chromatic
  number of a graph is the minimum number of colors used by a locally
  identifying vertex-coloring. In this paper, we prove that for any
  graph class of bounded expansion, the lid-chromatic number is
  bounded. Classes of bounded expansion include minor closed classes
  of graphs. For these latter classes, we give an alternative proof to
  show that the lid-chromatic number is bounded. This leads to an
  explicit upper bound for the lid-chromatic number of planar graphs.
  This answers in a positive way a question of Esperet~\emph{et al.}
  [L.~Esperet, S.~Gravier, M.~Montassier, P.~Ochem and A.~Parreau.
  Locally identifying coloring of graphs. \emph{Electronic Journal of
    Combinatorics}, {\bf 19(2)}, 2012.].
\end{abstract}  

\section{Introduction}
A vertex-coloring is said to be \emph{locally identifying} if
$(i)$ the vertex-coloring is proper (i.e. no adjacent vertices receive
the same color), and $(ii)$ for any adjacent vertices $u,v$, the set of
colors assigned to the closed neighborhood of $u$ differs from the set
of colors assigned to the closed neighborhood of $v$ whenever these
neighborhoods are distinct. The \emph{locally identifying chromatic number} of
the graph $G$ (or lid-chromatic number, for short), denoted by
$\chi_{lid}(G)$, is the smallest number of colors required in any
locally identifying coloring of $G$.

Locally identifying colorings of graphs have been recently introduced
by Esperet et al.~\cite{EGMOP12} and later studied by Foucaud et
al.~\cite{FHLPP12}. They are related to identifying codes
\cite{KCL98,lobs}, distinguishing colorings \cite{BRS03,BS97,CHS96}
and locating-colorings~\cite{CEHSZ02}. For example, upper bounds on
lid-chromatic number have been obtained for bipartite graphs,
$k$-trees, outerplanar graphs and bounded degree graphs.
An open question asked by Esperet et al.~\cite{EGMOP12} was to know whether
$\chi_{lid}$ is bounded for the class of planar graphs. In this paper,
we answer positively to this question proving more generally that
$\chi_{lid}$ is bounded for any class of bounded expansion.

In Section~\ref{sec:treedepth}, we first give a tight bound of
$\chi_{lid}$ in term of the tree-depth. Then we use the fact that any
class of bounded expansion admits a low tree-depth coloring (that is a
$k$-coloring such that each triplet of colors induces a graph of
tree-depth $3$, for some constant $k$) to prove that it has bounded
lid-chromatic number.

In Section~\ref{sec:minor}, we focus on minor closed classes of graphs
which have bounded expansion and give an alternative bound on the
lid-chromatic number, which gives an explicit bound for planar graphs.

The next section is devoted to introduce notation and preliminary
results.

\section{Notation and preliminary results}

Let $G=(V,E)$ be a graph.  For any vertex $u$, we denote by $N_G(u)$
its \emph{neighborhood} in $G$ and by $N_G[u]$ its \emph{closed
  neighborhood} in $G$ ($u$ together with its adjacent vertices). The
notion of neighborhood can be extended to sets as follows: for
$X\subseteq V$, $N_G[X] = \set{w \in V(G) \mid \exists v\in X, w\in
  N[v]}$ and $N_G(X) = N_G[X]\setminus X$.  When the considered graph
is clearly identified, the subscript is dropped.

The \emph{degree} of vertex $u$ is the size of its neighborhood.  The
\emph{distance} between two vertices $u$ and $v$ is the number of
edges in a shortest path between $u$ and $v$.  For $X\subseteq V$, we
denote by $G[X]$ the subgraph of $G$ \emph{induced by} $X$.

We say that two vertices $u$ and $v$ are \emph{twins} if $N[u]=N[v]$
(although they are often called \emph{true twins} in the literature,
we call them \emph{twins} for convenience).  In particular, $u$ and
$v$ are adjacent vertices. Note that if $u$ and $v$ are adjacent but
not twins, there exists a vertex $w$ which is adjacent to exactly one
vertex among $\{u,v\}$, i.e. $w\in N[u]\simdif N[v]$ (where $\Delta$
is the symmetric difference between sets). We say that $w$
\emph{distinguishes} $u$ and $v$, or simply $w$ \emph{distinguishes}
the edge $uv$. For a subset $X\subseteq V$, we say that a subset
$Y\subseteq V$ \emph{distinguishes} $X$ if for every pair $u,v$ of
non-twin vertices of $X$, there exists a vertex $w\in Y$ that
distinguishes the edge $uv$.  

Let $c:V\to \mathbb N$ be a vertex-coloring of $G$. The coloring $c$
is \emph{proper} if adjacent vertices have distinct colors. We denote
by $\chi(G)$ the \emph{chromatic number of $G$}, i.e. the minimum
number of colors in a proper coloring of $G$.  For any $X\subseteq V$,
let $c(X)$ be the set of colors that appear on the vertices of $X$. A
\emph{locally identifying coloring} (lid-coloring for short) of $G$ is
a proper vertex-coloring $c$ of $G$ such that for any two adjacent
vertices $u$ and $v$ that are not twins (i.e. $N[u]\ne N[v]$), we have
$c(N[u])\ne c(N[v])$. A graph $G$ is \emph{$k$-lid-colorable} if it
admits a locally identifying coloring using at most $k$ colors and the
minimum number of colors needed for any locally identifying coloring
of $G$ is the \emph{locally identifying chromatic number}
(lid-chromatic number for short) denoted by $\chi_{lid}(G)$. For a vertex
$u$, we say that $u$ \emph{sees} color $a$ if $a\in c(N[u])$. 
 For two adjacent vertices $u$ and $v$, a color that is in the set
$c(N[u])\simdif c(N[v])$ \emph{separates} $u$ and $v$, or simply
\emph{separates} the edge $uv$.  The notion of chromatic number (resp.
lid-chromatic number) can be extended to a class of graphs $\clac$ as
follows: $\chi(\clac) = \sup\set{\chi(G), G\in \clac}$ (resp.
$\chi_{lid}(\clac) = \sup\set{\chi_{lid}(G), G\in \clac}$).

The following theorem is due to Bondy~\cite{B72}: 
\begin{theorem}[Bondy's theorem \cite{B72}]
  Let $\mathcal A=\{A_1,...,A_n\}$ be a collection of $n$ distinct
  subsets of a finite set $X$. There exists a subset $X'$ of $X$ of
  size at most $n-1$ such that the sets $A_i\cap X'$ are all distinct.
\end{theorem}

\begin{corollary}\label{cor:bondy}
  Let $C$ be a $n$-clique subgraph of $G$. There exists a vertex
  subset $S(C)\subseteq V(G)$ of size at most $n - 1$ that
  distinguishes all the pair of non-twin vertices of $C$.
\end{corollary}

\begin{proof}
  Let $C$ be a $n$-clique subgraph of $G$ induced by the vertex set
  $V(C) = \set{v_1, v_2, \ldots,v_n}$. Let $\mathcal A=\set{N[v_i]
    \mid v_i\in V(C)}$ be a collection of distinct subsets of the
  finite set $X=\bigcup_{1\le i \le n}N[v_i]$.  Note that some $v_i$'s
  might be twins in $G$ (i.e. $N[v_i]=N[v_j]$ for some $v_i,v_j\in
  V(C)$) and therefore $|\mathcal A|$ could be smaller than $n$. By
  Bondy Theorem, there exists $S(C)\subseteq X$ of size at most
  $|\mathcal A| - 1 \le n-1$ such that for any distinct elements
  $A_1,A_2$ of $\mathcal A$, we have $A_1 \cap S(C) \neq A_2 \cap
  S(C)$.
  
  Let us prove that $S(C)$ is a set of vertices that distinguish all
  the pairs of non-twin vertices of $C$. For a pair of non-twin vertices
  $v_i,v_j$ of $C$, we have $N[v_i]\neq N[v_j]$. By definition of $S(C)$,
  we have $N[v_i] \cap S(C) \neq N[v_j] \cap S(C)$, then there exists
  $w\in S(C)$ that belongs to $N[v_i]\simdif N[v_j]$.  Therefore, $w$
  distinguishes the edge $v_iv_j$.
\end{proof}

\section{Bounded expansion classes of graphs}\label{sec:treedepth}

A rooted tree is a tree with a special vertex, called the \emph{root}.
The \emph{height} of a vertex $x$ in a rooted tree is the number of
vertices on a path from the root to $x$ (hence, the height of the root
is 1). The \emph{height} of a rooted tree $T$ is the maximum height of
the vertices of $T$. If $x$ and $y$ are two vertices of $T$, $x$ is an
\emph{ancestor} of $y$ in $T$ if $x$ belongs to the path between $y$
and the root. The \emph{closure} $\clos(T)$ of a rooted tree $T$ is
the graph with vertex set $V(T)$ and edge set $\{xy \ | \ x\text{ is
  an ancestor of }y\text{ in }T, x\neq y\}$.  The \emph{tree-depth}
$\td(G)$ of a connected graph $G$ is the minimum height of a rooted
tree $T$ such that $G$ is a subgraph of $\clos(T)$. If $G$ is not
connected, the tree-depth of $G$ is the maximum tree-depth of its
connected components.

Let $p$ be a fixed integer. A \emph{low tree-depth coloring} of a
graph $G$ (relatively to $p$) is a coloring of the vertices of $G$
such that the union of any $i\leq p$ color classes induces a graph of
tree-depth at most $i$. Let $\chi^{\td}_p(G)$ be the minimum number of
colors required in such a coloring. Note that as tree-depth one graphs and 
tree-depth two graphs are respectively the stables and star forests, 
$\chi^{\td}_1$ and $\chi^{\td}_2$ respectively correspond to the usual
chromatic number and the star chromatic number.
\medskip

In the following of this section, we first give a tight bound on the
lid-chromatic number in terms of tree-depth.

\begin{proposition}\label{prop:td}
For any graph $G$,  $\chi_{lid}(G)\leq 2\td(G)-1$ and this is tight.
\end{proposition}

Using this bound, we then bound the lid-chromatic number
in terms of $\chi^{\td}_3$.

\begin{theorem}\label{thm:chlidch3}
For any graph $G$, $$\chi_{lid}(G)\leq 6^{\chi^{\td}_3(G) \choose 3}.$$
\end{theorem}

Classes of graphs of bounded expansion have been introduced by Ne\v
set\v ril and Ossona de Mendez~\cite{NO08}. These classes contain
minor closed classes of graphs and any class of graphs defined by an
excluded topological minor. Actually, these classes of graphs are
closely related to low tree-depth colorings:

\begin{theorem}[Theorem 7.1~\cite{NO08}]\label{thm:nesetril-pom}
  A class of graphs $\clac$ has bounded
  expansion if and only if $\chi^{\td}_p(\clac)$ is bounded for
  any~$p$.
\end{theorem}

We therefore deduce the following corollary from
Theorems~\ref{thm:chlidch3} and~\ref{thm:nesetril-pom}: 

\begin{corollary}\label{cor:bounded}
  For any class $\clac$ of bounded expansion, 
  $\chi_{lid}(\clac)$ is bounded.
\end{corollary}

It is in particular true for a class of bounded tree-width. A consequence is that $\chi_{lid}$ is bounded for chordal graphs by a function of the clique number (which is equals to the tree-width plus 1 for a chordal graph). It is conjectured by Esperet {\em et al.} \cite{EGMOP12} that $\chi_{lid}(G)\leq 2\omega(G)$ if $G$ is chordal.

\medskip

We now prove Proposition~\ref{prop:td}.

\begin{proofof}{Proposition~\ref{prop:td}}
  Let us first prove that the bound is tight. Consider the graph
  $H_n$ obtained from a complete graph, with vertex set $\{a_1,\ldots,
  a_n\}$, by adding a pendant vertex $b_i$ to every $a_i$ but one, say
  for $1\le i< n$. The tree-depth of this graph is at least $n$ as it
  contains a $n$-clique. Indeed, given a rooted tree $T$, two vertices
  at the same height are non-adjacent in $clos(T)$, we thus need at
  least $n$ levels. Actually the tree-depth of this graph is at most
  $n$ since the tree $T$ rooted at $a_1$, and such that $a_i$ has two sons
  $a_{i+1}$ and $b_i$, for $1\le i< n$, has height $n$ and is such that $clos(T)$ 
  contains $H_n$ as a subgraph.

  Let us show that in any lid-coloring of $H_n$ all the vertices must
  have distinct colors, and thus use $2n-1 = 2 \td(H_n) - 1$
  colors. Indeed, two vertices $a_i$ must have different colors as the
  coloring is proper. A vertex $b_j$ cannot use the same color as a
  vertex $a_i$, as otherwise the vertex $a_j$ would only see the $n$
  colors used in the clique, just as $a_n$. Similarly if two vertices
  $b_i$ and $b_j$ would use the same color, the vertices $a_i$ and
  $a_j$ would see the same set of colors.

  Let us now focus on the upper bound.  We prove the result for a
  connected graph and by induction on the tree-depth of $G$, denoted
  by $k$.  The result is clear for $k=1$ (the graph is a single
  vertex).
  
  Let $G$ be a graph of tree-depth $k>1$ and let $T$ be a rooted tree
  of height $k$ such that $G$ is a subgraph of $\clos(T)$.  If $T$ is
  a path, the result is clear since there are only $k$ vertices.  So
  assume that $T$ is not a path, and let $r$ be the root of $T$.  Let
  $s$ be the smallest height such that there are at least two vertices
  of height $s+1$. We name $r_i$, for $i\in\set{1,...,s}$, the unique
  vertex of height $i$.  Let $R=\{r_1,...,r_{s}\}$.  Note that each of
  the vertices of $R$ is adjacent to all the vertices of $\clos(T)$.
  Therefore, we can choose the way we label the $s$ vertices in $R$
  (i.e. we can choose the height of each of them in $T$) without
  changing $\clos(T)$.
  
  Necessarily, $G\setminus R$ has at least two connected components.
  Let $G_1,\ldots,G_{\ell}$ be its connected components and thus
  $\ell\ge 2$. We choose $T$ such that $s$ is minimal. It implies that
  for each $i\in \{1,...,s\}$, $r_i$ has neighbors in all the
  components $G_1$,...,$G_{\ell}$. Indeed, if it is not the case, by
  permuting the elements of $R$ (this is possible by the above
  remark), we can assume without loss of generality that $r_{s}$ does
  not have a neighbor in $G_{\ell}$. Therefore, the set of edges
  $e(r_{s},G_\ell) = \{r_{s}x : x\in V(G_\ell)\}$ of $\clos(T)$ are
  not used by $G$.  Then let $T'$ be the tree obtained from $T$ by
  moving the whole component $G_{\ell}$ one level up in such a way
  that the root of the subtree corresponding to $G_{\ell}$ is now the
  son of $r_{s-1}$ (instead of $r_s$ previously). Note that
  $\clos(T')$ is isomorphic to $\clos(T) \setminus e(r_{s},G_\ell)$
  and thus $G$ is a subgraph of $\clos(T')$.  This new tree $T'$ has
  two vertices at height $s$, contradicting the minimality of $s$.
  
  Any connected component $G_j$ has tree-depth at most $k'=k-s<k$.
  By induction, for each $j\in \{1,...,\ell\}$, there exists a
  lid-coloring $c_j$ of $G_j$ using colors in
  $\set{1,\ldots,2k'-1}$. For each $c_j$, there is a minimum value
  $s_j$ such that every vertex $r_i$ sees a color in
  $\set{1,\ldots,s_j}$ in $G_j$. We choose a $(2k'-1)$-lid-coloring
  $c_j$ of $G_j$ such that $s_j$ is minimized. Note that for each
  color $a\le s_j$, there exists $r_i\in R$ such that $r_i$ sees color
  $a$ in $G_j$ but no other color of $\set{1,\ldots,s_j}$. Otherwise,
  after permuting colors $a$ and $s_j$, every vertex $r_i\in R$ would
  see a color in $\set{1,\ldots,s_{j}-1}$, contradicting the
  minimality of $s_j$. Assume without loss
  of generality that $s_1 \ge s_2\ge \ldots \ge s_\ell$.
  
  We replace in $c_1$ the colors $1,2,\ldots, s_1$ by
  $1',2',\ldots,s'_1$. Note that now each vertex $r_i$ sees a color in
  $\set{1',\ldots,s'_1}$ (in $G_1$) and a color in
  $\set{1,\ldots,s_2}$ (in $G_2$). Furthermore, the other vertices of
  $G$ (that is the vertices in $G_1,\ldots,G_\ell$) do not have this
  property since $s_1\ge s_2$. Thus at this step every edge $xr_i$
  with $x$ in some $G_j$ is separated.
  
  Now we color each vertex $r_i$ with color $i^*$.  Let $c : V(G) \to
  \set{1^*,\ldots,s^*}\cup\set{1',\ldots,s'_1}\cup\set{1,\ldots,2k'-1}$
  be the current coloring of $G$.
  
  Note that now every distinguishable edge $xy$ in some $G_j$ is
  separated. Indeed, either $xy$ was distinguished in $G_j$ and it has
  been separated by $c_j$, or $xy$ is distinguished by some $r_i$ and
  it is separated by the color $i^*$. Note also that $c$ is a proper coloring.

  It remains to deal with the edges $r_ir_j$. For that purpose we will
  refine some color classes. In the following lemma we show that such
  refinements do not damage what we have done so far.

  \begin{claim}
    Consider a graph $G$ and a coloring $\varphi\ :\ V(G)
    \longrightarrow \set{1,\ldots,k}$.  Consider any refinement
    $\varphi'$ of $\varphi$, obtained from $\varphi$ by recoloring
    with color $k+1$ some vertices colored $i$, for some $i$. Any edge
    $xy$ of $G$ properly colored (resp. separated) by $\varphi$ is
    properly colored (resp. separated) by $\varphi'$.
  \end{claim}
  Indeed if $\varphi(x)\neq \varphi(y)$ then $\varphi'(x)\neq
  \varphi'(y)$, and if $i\in \varphi(N[x])\simdif \varphi(N[y])$ then
  $i$ or $k+1 \in \varphi'(N[x])\simdif \varphi'(N[y])$.
  
  Let us define a relation $\mathcal R$ among vertices in $R$ by
  $r_i\mathcal R r_j$ if and only if $c(N[r_i])=c(N[r_j])$. Let
  $R_1$,...,$R_{\bar{s}}$ be the equivalence classes of the relation
  $\mathcal R$ (note that each $R_i$ forms a clique since every
  $r_i$ has distinct colors). We have $\bar{s}\geq s_1$. Indeed, by
  definition of $s_1$ and the coloring $c_1$, for each color $a\in
  \set{1',\ldots, s'_1}$, there exists $r_i\in R$ that sees $a$ in
  $G_1$ but no other color of $\set{1',\ldots,s'_1}$. This vertex
  $r_i$ belongs to some equivalence class $R_j$ and thus all the
  vertices of $R_j$ sees color $a$ in $G_1$ but no other color of
  $\set{1',\ldots,s'_1}$.

  By Corollary~\ref{cor:bondy}, there is a
  vertex set $S(R_i)$ of size at most $|R_i|-1$ which distinguishes
  all pairs of non-twin vertices in $R_i$. We give to the vertices of
  $S(R_i)$ new distinct colors. By the previous claim, this last operation
  does not damage the coloring, and now all the distinguishable edges
  are separated.
  
  Since for this last operation we need $s-\bar{s}$ new colors, since
  we used $2k'-1$ colors $\set{1,\ldots,2k'-1}$, $s_1$ colors
  $\set{1',\ldots,s_1'}$ and $s$ 
  colors $\set{1^*,\ldots,s^*}$, the total number of colors is $ (s-\bar{s}) +
  (2k'-1) + s_1 + s = 2k -1 +s_1 -\bar{s} \le 2k-1$. This concludes
  the proof of the theorem.
\end{proofof}

We are now ready to prove Theorem~\ref{thm:chlidch3}:

\begin{proofof}{Theorem~\ref{thm:chlidch3}}
  Let $\alpha$ be a low tree-depth coloring of $G$ with parameter
  $p=3$ and using $\chi^{\td}_3(G)$ colors. Let
  $A=\{\alpha_1,\alpha_2,\alpha_3\}$ be a triplet of three distinct
  colors and let $H_A$ be the subgraph of $G$ induced by the vertices
  colored by a color of $A$. Since $H_A$ has tree-depth at most $3$, by
  Proposition \ref{prop:td}, $H_A$ admits a lid-coloring $c_A$ with five colors
  (says colors $1$ to $5$). We extend $c_A$ to the whole graph by
  giving color $0$ to the vertices in $V(G)\setminus V(H_A)$.
  
  Let $A_1, A_2, \ldots, A_k$ be the $k={\chi^{\td}_3(G) \choose 3}$
  distinct triplets of colors.
  We now construct a coloring $c$ of $G$ giving to each vertex $x$ of
  $G$ the $k$-uplet $$(c_{A_1}(x),c_{A_2}(x),\ldots,c_{A_k}(x)).$$
  
  The coloring $c$ is using $6^k$ colors.
  Clearly it is a proper coloring: each pair of adjacent vertices will
  be in some common graph $H_A$ and will receive distinct colors in
  this graph.  Let $x$ and $y$ be two adjacent vertices with $N[x]\neq
  N[y]$. Let $w$ be a vertex adjacent to only one vertex among $x$ and
  $y$. Let $A=\{\alpha(x), \alpha(y), \alpha(w)\}$.  Vertices $x$ and
  $y$ are not twins in the graph $H_A$. Hence $c_A(N[x])\neq
  c_A(N[y])$ and therefore, $c(N[x])\neq c(N[y])$.
\end{proofof}

\section{Minor closed classes of graphs}\label{sec:minor}

Let $G$ and $H$ be two graphs. $H$ is a \emph{minor} of $G$ if $H$ can
be obtained from $G$ with successive edge deletions, vertex deletions
and edge contractions. A class $\clac$ is \emph{minor closed} if for
any graph $G$ of $\clac$, for any minor $H$ of $G$, we have $H\in
\clac$. The class $\clac$ is \emph{proper} if it is not the class of
all graphs. Let $H$ be a graph. A \emph{$H$-minor free graph} is a
graph that does not have $H$ as a minor. We denote by $\clak_n$ the
$K_n$-minor-free class of graphs. It is clear that any proper minor
closed class of graphs is included in the class $\clak_n$ for some
$n$.  It is folklore that any proper minor closed class of graphs
$\clac$ has a bounded chromatic number $\chi(\clac)$.

The class of graphs of bounded expansion includes all the proper minor
closed classes of graphs. Thus, by Corollary~\ref{cor:bounded}, proper
minor closed classes have bounded lid-chromatic number. In this
section, we focus on these latter classes and give an alternative
upper bound on the lid-chromatic number. This gives us an explicit
upper bound for the lid-chromatic number of planar graphs.

Consider any proper minor closed class of graphs $\clac$. Since
$\clac$ is proper, there exists $n$ such that $\clac$ does not contain
$K_n$, that is $\clac\subseteq\clak_n$. Let $\cla{C}^N$ be the class of
graphs defined by $H\in \cla{C^N}$ if and only if there exists
$G\in\clac$ and $v\in G$ such that $H=G[N(v)]$. Note that $\cla{C^N}$
is a minor-closed class of graphs. Indeed, given any $H\in\cla{C^N}$,
let $G\in\clac$ and $v\in V(G)$ such that $H=G[N(v)]$. Let
$H'$ be any minor of $H$.  Since $\clac$ is minor-closed and $H$ is a
subgraph of $G$, there exists a minor $G'$ of $G$ such that $H' =
G'[N(v)]$. Therefore, $H'$ belongs to $\cla{C^N}$.

We prove the following result on minor-closed classes of graphs:

\begin{theorem}\label{thm:induction}
  Let $\clac$ be a proper minor closed class of graphs and let
  $n\ge 3$ be such that $\clac \subseteq \clak_n$. Then
  $$\chi_{lid}(\clac) \leq 4\cdot\chi_{lid}(\cla{C^N})\cdot\chi(\clac)^{n-3}$$
\end{theorem}

The class of trees is exactly the class $\clak_3$. Esperet et al.
\cite{EGMOP12} proved the following result.

\begin{proposition}[\cite{EGMOP12}]\label{prop:tree}
$\chi_{lid}(\clak_3)\leq 4$.
\end{proposition}

It is clear that $\clak_3^N$ is the class of stable graphs and therefore,
$\chi_{lid}(\clak_3^N) = 1$. Note that Theorem~\ref{thm:induction} implies
Proposition~\ref{prop:tree}.

Assume that $\chi_{lid}(\clak_{n-1})$ is bounded for some $n\ge 4$. It
is clear that $\clak_{n}^N = \clak_{n-1}$. Then, by
Theorem~\ref{thm:induction}, we have $\chi_{lid}(\clak_n) \leq
4\cdot\chi_{lid}(\clak_{n-1})\cdot\chi(\clak_n)^{n-3}$. Since
$\chi_{lid}(\clak_{n-1})$ and $\chi(\clak_n)$ are bounded,
$\chi_{lid}(\clak_n)$ is bounded.

Esperet et al. \cite{EGMOP12} also proved the following result.

\begin{proposition}[\cite{EGMOP12}]\label{prop:outer}
If $G$ is an outerplanar graph, $\chi_{lid}(G)\leq 20$.
\end{proposition}

We can then deduce from Theorem~\ref{thm:induction} and
Proposition~\ref{prop:outer} the following corollary:

\begin{corollary}
 Let $\cla{P}$ be the class of planar graphs. Then
 $\chi_{lid}(\cla{P})\leq 1280$. 
\end{corollary}

\begin{proof}
  Any graph $G\in\cla{P}$ is $\set{K_{3,3},K_5}$-minor free and thus
  $\cla{P}$ is a proper minor closed class of graphs. Moreover, the
  neighborhood of any vertex of $G\in\cla{P}$ is an outerplanar graph.
  By Proposition \ref{prop:outer}, we have $\chi_{lid}(\cla{P}^N)\leq
  20$.  Furthermore, the Four-Color-Theorem gives $\chi(\cla{P}) = 4$.
  By Theorem \ref{thm:induction}, $\chi_{lid}(\cla{P})\leq 4\times 20
  \times 4^2=1280$.
\end{proof}

We finally give the proof of Theorem~\ref{thm:induction}.

\begin{proofof}{Theorem~\ref{thm:induction}}
  Let $G\in \clac$ and let $u$ be a vertex of minimum degree. For any
  $i$, define $\Vui$ as the set of vertices of $G$ at distance exactly
  $i$ from $u$ and let $\Gui = G[\Vui]$. Let $s$ be the largest
  distance from a vertex of $V$ to $u$. In other words, there are
  $s+1$ nonempty sets $\Vui$ (note that $V_{\{u,0\}} = \set{u}$).
  
  For any $i$, contracting in $G$ the subgraph
  $G[\V{u}{0}\cup\V{u}{1}\cup\ldots\cup\V{u}{i-1}]$ in a single vertex
  $x$ gives a graph $G'\in \clac$ such that $x$ is exactly adjacent to
  every vertex of $\Gui$. Therefore, for any $i$, $\Gui\in\cla{C^N}$.
  Hence, $\chi_{lid}(\Gui) \le \chi_{lid}(\cla{C^N})$ for any $i$.
  Moreover, $\cla{C^N} \subseteq \clak_{n-1}$. Indeed, suppose that
  there exists $H\in\cla{C^N}$ that admits $K_{n-1}$ as a minor.
  Therefore there exists $G\in \clac$ such that $H=G[N(v)]$ for some
  $v\in G$.  Taking $v$ together with its neighborhood would give
  $K_n$ as a minor, that contradicts the fact that
  $\clac\subseteq\clak_n$. Hence, any $\Gui\in\clak_{n-1}$.
  
  We construct a lid-coloring of $G$ using
  $4\cdot\chi_{lid}(\cla{C^N})\cdot\chi(\clac)^{n-3}$ colors. This coloring is
  constructed with three different colorings of the vertices of $G$:
  $c_1$ which uses $4$ colors, $c_2$ which uses
  $\chi_{lid}(\cla{C^N})$ colors and $c_3$ which is itself composed of
  $n-3$ colorings with $\chi(\clac)$ colors. The final color $c(v)$ of
  a vertex $v$ will be the triplet $(c_1(v),c_2(v),c_3(v))$.  Hence
  the coloring $c$ uses at most
  $4\chi_{lid}(\cla{C^N})\chi(\clac)^{n-3}$ colors. The coloring $c_1$
  is used to separate the pairs of vertices that lie in distinct sets
  $\Vui$. The coloring $c_2$ separates the pairs of vertices that lie
  in the same set $\Vui$ and are not twins in $\Gui$. Finally, the
  coloring $c_3$ separates the pairs of vertices that lie in the same
  set $\Vui$, that are twins in $\Gui$ but that are not twins in $G$.

The coloring $c_1$ is simply defined by $c_1(v)\equiv i\bmod 4$ if
$v\in \Vui$.  

To define $c_2$, we define for each $i$, $0\leq i \leq s$, a
lid-coloring $c^i_{2}$ of $\Gui$ using colors $1$ to
$\chi_{lid}(\cla{C^N})$. Then $c_2$ is defined by $c_2(v)=c^i_{2}(v)$
if $v\in \Vui$.

We now define the coloring $c_3$. Let $\Vui^{id}$ be the set of
vertices of $\Vui$ that have a twin in $\Gui$:
$$\Vui^{id}=\{ v \in \Vui \ |\ \exists w \in \Vui, N_{\Gui}[v]= N_{\Gui}[w]\}.$$

Let $\Gui^{id} = \Gui[\Vuid]$. Since the relation ``be twin'' is
transitive (i.e. if $u$ and $v$ are twins, and $v$ and $w$ are twins,
then $u$ and $w$ are twins), then $\Gui^{id}$ is clearly a union of
cliques. In addition, since $\Gui\in\clak_{n-1}$, the
connected components of $\Gui^{id}$ are cliques of size at most $n-2$.

Let $C$ be a clique of $\Gui^{id}$. 
By Corollary~\ref{cor:bondy}, there exists a subset $S(C)\subseteq V(G)$
of at most $n-3$ vertices that distinguishes all the pairs of non-twin
vertices of $C$. Note that by definition of $C$, $S(C)\cap \Vui = \emptyset$, 
and thus $S(C) \subseteq V_{u,i-1}\cup V_{u,i+1}$.

Let $\mathcal S = \{(v,C) \mid v \in S(C) \text{ and $C$ is a clique in
  a graph $\Gui^{id}$}\}$.
We partition $\mathcal S$ in $s\times(n-3)$ sets $S_i^{k}$, $1\leq i
\leq s$, $1\leq k \leq n-3$, such that:
\begin{itemize}
\item if $(v,C) \in S_i^k$ for some $k$, then $v\in \Vui$; 
\item if $(v,C)$ and $(w,C')$ are two elements of $S_i^k$, then $C\neq C'$.
\end{itemize}
This partition can be done because each set $S(C)$ has size at most
$n-3$.

For each $S_i^k = \set{(x_1,C_1),(x_2,C_2),\ldots,(x_t,C_t)}$, we
define a graph $H_i^k$ as follows.  We start from the graph induced by
$\Vui \cup V(C_1) \cup V(C_2) \cup \ldots \cup V(C_t)$. Then,
for each $(x_j,C_j)$ in $S_i^k$, we contract $C_j$ in a single vertex
$y_j$ and finally, we contract the edge $x_jy_j$ on the vertex $x_j$. Note that
$\Vui$ is the vertex set of $H_i^k$.
Note also that $H_i^k\in\clac$ since it is obtained from a subgraph of $G$ 
by successive edge-contractions. Therefore, $\chi(H^i_k) \leq
\chi(\clac)$. 

We now define a proper coloring $c_3^{i,k}$ of $H_i^k$ with colors $1$
to $\chi(\clac)$.  Let $c_3^k$ be the coloring of
vertices of $G$ defined by $c_3^k(v)=c_3^{i,k}(v)$ if $v\in \Vui$.
Finally, $c_3$ is defined by $c_3(v)=(c_3^1(v),\ldots,c_3^{n-3}(v))$, and
the final color of $v$ is $c(v)=(c_1(v),c_2(v),c_3(v))$.

We now prove that $c$ is a lid-coloring of $G$. First, $c$ is a proper
coloring. Indeed, two adjacent vertices that are not in the same set
$\Vui$ lie in consecutives set $\Vui$ and $\V{u}{i+1}$ and thus have
different colors in $c_1$, and two adjacent vertices in the same set
$\Vui$ have different colors in $c_2$ (which induces a proper coloring
on $\Vui$).

Let now $x$ and $y$ be two adjacent vertices with $N[x]\neq N[y]$. We
will prove that $c(N[x])\neq c(N[y])$. We distinguish three cases.

\begin{enumerate}[C{a}se 1:]
  \item $x\in \Vui$ and $y\in \V{u}{i+1}$.
    
    If $x=u$, then $y$ has a neighbor $v$ in $\V{u}{i+2}=\V{u}{2}$.
    Indeed, $u$ is taken with minimum degree, so $y$ has at least as
    many neighbors as $u$ and does not have the same neighborhood
    than $u$, implying that $y$ has a neighbor in $\V{u}{2}$.  Then
    $c_1(v)=2\notin c_1(N[u])$ and so $c(N[x])\neq c(N[y])$.
    
    Otherwise, $x$ has neighbor $v$ in $\V{u}{i-1}$ and $c_1(v)\equiv
    i-1 \pmod 4 \in c_1(N[x])$. On the other hand, all the neighbors
    of $y$ belong to $\Vui \cup \V{u}{i+1} \cup \V{u}{i+2}$ and
    therefore $c_1(N[y]) \subseteq \{i,i+1,i+2 \pmod 4\}$. Thus,
    $c(N[x])\neq c(N[y])$.

\item $x$ and $y$ belong to $\Vui$ and they are not twins in $\Vui$
  (i.e. $N_{\Vui}[x] \neq N_{\Vui}[y]$).
  
  By definition of the coloring $c_2^i$, there exists a color $a$ that
  separates $x$ and $y$, i.e. $a \in c_2^i(N_{\Vui}[x])\simdif
  c_2^i(N_{\Vui}[y])$. Then we necessarily have $c(N[x])\neq c(N[y])$. 

\item $x$ and $y$ belong to $\Vui$ and they are twins in $\Vui$
  (i.e. $N_{\Vui}[x] = N_{\Vui}[y]$).
  
  In this case, vertices $x$ and $y$ are in the set $\Vui^{id}$. Let
  $C$ be the clique of $\Gui$ containing $x$ and $y$. Let $v\in S(C)$
  that distinguishes $x$ and $y$; thus, $v\in
  \V{u}{j}$ for $j=i-1$ or $j=i+1$. Wlog, $v\in N[x]$ but $v\notin
  N[y]$. Let $S_j^k$ be the part of $\mathcal S$ that contains
  $(v,C)$. Suppose that there exists a neighbor $w$ of $y$ such that
  $c(v)=c(w)$. Then $w$ lies in $\V{u}{j}$ because of the coloring
  $c_1$. However, in the graph $H_j^k$, the vertex $v$ is adjacent to
  all the neighbors of $y$ in $\V{u}{j}$, and in particular is
  adjacent to $w$; therefore, $c_3^{j,k}(v) \neq c_3^{j,k}(w)$, a
  contradiction. Therefore, the vertex $y$ does not have any neighbor
  that has the same color as $v$. Hence, $c(v)\notin c(N[y])$, and
  $c(N[x])\neq c(N[y])$.
\end{enumerate}
\end{proofof}

\end{document}